\documentclass[leqno]{amsart}
\usepackage{amsmath}
\usepackage{amssymb}
\usepackage{amsthm}
\usepackage{enumerate}
\usepackage[mathscr]{eucal}
\usepackage{mathrsfs}
\usepackage{xcolor}
\theoremstyle{plain}
\usepackage{tikz}
\newtheorem{theorem}{Theorem}[section]

\newtheorem{prop}[theorem]{Proposition}
\theoremstyle{definition}
\newtheorem{definition}[theorem]{Definition}
\newtheorem{remark}[theorem]{Remark}

\newtheorem{cor}[theorem]{Corollary}
\theoremstyle{remark}

\begin{document}

\title [An approximation problem in the space of bounded operators]{An approximation problem in the space of bounded operators}

\author[Arpita Mal]{Arpita Mal}

\address[ ]{
	 Kolkata 700032\\ India.}
\email{arpitamalju@gmail.com}


\thanks{The author would like to thank Prof. Kallol Paul for his valuable comments on this paper. }




\date{}
\maketitle
\begin{abstract}
 For Banach spaces $X,Y,$ we consider a distance problem in the space of bounded linear operators $\mathcal{L}(X,Y).$ Motivated by a recent paper \cite{RAO21}, we obtain sufficient conditions so that for a compact operator $T\in\mathcal{L}(X,Y)$ and a closed subspace $Z\subset Y,$ the following equation holds, which relates global approximation with local approximation:
 \[d(T,\mathcal{L}(X,Z))=\sup\{d(Tx,Z):x\in X,\|x\|=1\}.\]
 In some cases, we show that the supremum is attained at an extreme point of the corresponding unit ball. Furthermore, we obtain some situations when the following equivalence holds:
 $$T\perp_B \mathcal{L}(X,Z)\Leftrightarrow T^{**}x_0^{**}\perp_B Z^{\perp\perp}\Leftrightarrow T^{**}\perp_B\mathcal{L}(X^{**},Z^{\perp\perp}),$$ for some $x_0^{**}\in X^{**}$ satisfying $\|T^{**}x_0^{**}\|=\|T^{**}\|\|x_0^{**}\|,$ where $Z^\perp$ is the annihilator of $Z.$ One such situation is when $Z$ is an $L^1-$predual space and an $M-$ideal in $Y$ and $T$ is a multi-smooth operator of finite order. Another such situation is when $X$ is an abstract $L_1-$space and $T$ is a multi-smooth operator of finite order. Finally, as a consequence of the results, we obtain a sufficient condition for proximinality of a subspace $Z$  in $Y.$ 
\end{abstract}

\section{Introduction}
In this paper, our aim is to study distance problem and Birkhoff-James orthogonality in the space of bounded linear operators. We approximate the distance of a bounded linear operator from a subspace of operators with the distance of its image from a subspace. 
To state the problem, we first introduce necessary notations and terminologies.

In this paper, $X,Y$ denote real Banach spaces and $Z$ denotes a closed subspace of $Y.$ Let $S_X$ and $B_X$ denote the unit sphere and the unit ball of $X,$ i.e., $S_{X}=\{x\in X:\|x\|=1\}$ and $B_{X}=\{x\in X:\|x\|\leq 1\}.$ Suppose $E_X$ denotes the set of all extreme points of $B_X$ and $X^{*}$ denotes the dual space of $X.$ We always assume that $X$ is canonically embedded in its bidual $X^{**}.$ The symbol $\mathcal{L}(X,Y)~(\mathcal{K}(X,Y))$ denotes the space of all bounded (compact) linear operators from $X$ to $Y.$ For $T\in \mathcal{L}(X,Y),$ $M_T=\{x\in S_X:\|Tx\|=\|T\|\},$ the collection of all unit vectors of $X,$ at which $T$ attains its norm. For a non-zero element $x\in X,$ suppose $J(x)=\{x^*\in S_{X^*}:x^*(x)=\|x\|\}.$ Recall that $J(x)$ is a non-empty, weak*compact, convex subset of $S_{X^*}.$ The symbol $E_{J(T)}$ denotes the set of all extreme points of $J(T).$ A non-zero element $x\in X$ is said to be a multi-smooth point of finite order (or $k-$smooth) \cite{KS,LR} if $J(x)$ contains finitely many (exactly $k$) linearly independent functionals. In other words, $x$ is a multi-smooth point of finite order if $Span(J(x))$ is finite-dimensional. We say that $x$ is smooth if $J(x)$ is singleton and $X$ is smooth if $x$ is smooth for all non-zero $x\in X.$ Similarly, a non-zero operator $T\in\mathcal{L}(X,Y)$ is a multi-smooth operator of finite order if $Span(J(T))$ is finite-dimensional, where $J(T)=\{f\in S_{\mathcal{L}(X,Y)^*}:f(T)=\|T\|\}.$ In particular, $T$ is said to be a smooth operator if $J(T)$ is singleton. For the study of smooth and multi-smooth ($k-$smooth) operators see \cite{GY93,GRZA96,LR,MDP,MPD,MP20,MPRS,SPMR}.    \\
A classical problem in approximation theory is the distance problem. For $x\in X$ and a subspace $W\subset X,$ $d(x,W)=\inf_{w\in W}\|x-w\|$ is the distance of $x$ from $W.$  An element $w_0\in W$ is said to be a best approximation to $x$ out of $W$ if $\|x-w_0\|=d(x,W).$ We denote the collection of all best approximation(s) to $x$ out of $W$ by $\mathscr{L}_{W}(x),$ i.e., $\mathscr{L}_W(x)=\{w_0\in W:\|x-w_0\|=d(x,W)\}.$ A subspace $W\subset X$ is said to be a proximinal subspace of $X$ if $\mathscr{L}_W(x)\neq \emptyset$ for all $x\in X.$ Similarly, in operator space, for $T\in \mathcal{L}(X,Y)$ and a subspace $\mathscr{V}\subset \mathcal{L}(X,Y),$ $d(T,\mathscr{V})=\inf_{S\in \mathscr{V}}\|T-S\|$ and $\mathscr{L}_\mathscr{V}(T)=\{S_0\in \mathscr{V}:\|T-S_0\|=d(T,\mathscr{V})\}.$ For some recent study on best approximation and distance formula in operator spaces and $C^*$- algebra follow \cite{GS,MP22,Sb,SMP}. Note that, the notions of best approximation and Birkhoff-James orthogonality are closely related. For $x,y\in X,$ we say that $x$ is Birkhoff-James orthogonal \cite{B,J} to $y$ if $\|x+\lambda y\|\geq \|x\|$ for all scalars $\lambda$ and we denote it by $x\perp_B y.$ For a subspace $W\subset X,$ we say that $x\perp_B W$ if  $x\perp_B w$ for all $w\in W.$ It is now straightforward to check that $w_0\in \mathscr{L}_W(x)$ if and only if $x-w_0\perp_B W.$
Suppose that $K$ is a compact Hausdorff space. Let $C(K,Y)$ be the space of all continuous functions defined from $K$ to $Y$ equipped with the supremum norm. From a classical result \cite[Th. 2.4]{LC} by Light and Cheney, we know that if $f\in C(K,Y),$ then 
\[d(f,C(K,Z))=\sup_{k\in K}d(f(k),Z)=d(f(k_0),Z)\] for some $k_0\in K.$ The above distance formula provides a relation between global approximation and local approximation. Motivated by the above approximation result, recently in \cite{RAO21}, Rao studied analogous problem in the space of bounded linear operators. More precisely, he raised the question that for $T\in \mathcal{K}(X,Y),$ when the following minimax formula holds:
\begin{equation}\label{eq-minimax}
	d(T, \mathcal{L}(X,Z))=\sup_{x\in S_X}d(Tx,Z).
\end{equation} 
It is also interesting to ask when the supremum in (\ref{eq-minimax}) is attained. Whenever the the supremum is attained at some  $x_0\in S_X,$  we can further ask whether  $x_0$ can be chosen from $E_X.$  To answer these questions, we need to introduce a few notions. A Banach space $X$ is said to be an $L^1-$predual space \cite{L} if $X^*$ is isometrically isomorphic to $L^1(\mu)$ for a positive measure $\mu.$ A subspace $Z$ of $Y$ is said to be an $M-$ideal \cite{HWW93} in $Y,$ if there is a linear projection $P:Y^*\to Y^*$ such that  $\|y^*\|=\|Py^*\|+\|y^*-Py^*\|$ for all $y^*\in Y^*$ and $\ker(P)=Z^\perp,$ where $Z^\perp$ is the annihilator of $Z,$ i.e., $Z^\perp=\{y^*\in Y^*:y^*(z)=0~\forall~z\in Z\}.$\\
Note that, for any $x\in S_X$ and $S\in \mathcal{L}(X,Z),$ $\|Tx-Sx\|\geq d(Tx,Z).$ Thus, $$\|T-S\|=\sup_{x\in S_X}\|Tx-Sx\|\geq\sup_{x\in S_X}d(Tx,Z).$$ Now, taking infimum over $S\in \mathcal{L}(X,Z),$ we get
\begin{equation}\label{eq-inequal}
	d(T,\mathcal{L}(X,Z))\geq \sup_{x\in S_X}d(Tx,Z).
\end{equation}
In \cite[Th. 1]{RAO21} Rao proved the existence of the minimax formula (\ref{eq-minimax}) with the attainment of the supremum at an extreme point of $B_X$ assuming that $X$ is a reflexive, separable Banach space, $Z$ is an $L^1-$predual space and also an $M-$ideal in $Y.$ On the other hand, assuming that $X$ is a  separable Banach space, $Z$ is an $L^1-$predual space and an $M-$ideal in $Y,$ Rao in  \cite[Th. 6]{RAO21}	proved that under a local condition (a suitable smoothness condition) on $T$  the minimax formula  (\ref{eq-minimax}) holds. Both the proofs of these theorems are based on a lifting theorem from \cite[Th. II.2.1]{HWW93}.\\
Many researchers are devoted to the study of Birkhoff-James orthogonality in the space of operators (see \cite{BS,MP,SMP,SPM2} and the references therein). Note that, if there exists a vector $x_0\in M_T$ such that $Tx_0\perp_B Z,$ then for all $S\in \mathcal{L}(X,Z),$ $\|T-S\|\geq \|Tx_0-Sx_0\|\geq \|Tx_0\|=\|T\|.$ Therefore, in this case, $T\perp_B \mathcal{L}(X,Z).$ On the other hand, observe that if the supremum in (\ref{eq-minimax}) is attained at some $x_0\in S_X$ and $T\perp_B \mathcal{L}(X,Z),$ then from $\|T\|=d(T,\mathcal{L}(X,Z))=d(Tx_0,Z)\leq \|Tx_0\|=\|T\|,$ it follows that $d(Tx_0,Z)= \|Tx_0\|,$ i.e., $Tx_0\perp_B Z.$ In \cite{RAO21}, Rao used this approach to prove the implication $T\perp_B\mathcal{L}(X,Z)\Rightarrow Tx_0\perp_B Z$ for some $x_0\in M_T,$ whenever the supremum in (\ref{eq-minimax}) is attained.\\
In this paper, we prove the existence of formula (\ref{eq-minimax}) whenever $T$ is a multi-smooth operator of finite order and $T\perp_B \mathcal{L}(X,Z).$ In this case, we do not assume any restriction on the space $X.$ We only assume that $Z$ is an $L^1-$predual space and an $M-$ideal in $Y.$ Moreover, we prove the following equivalence
\begin{equation}\label{equiv}
	T\perp_B \mathcal{L}(X,Z)\Leftrightarrow T^{**}x_0^{**}\perp_B Z^{\perp\perp}\Leftrightarrow T^{**}\perp_B\mathcal{L}(X^{**},Z^{\perp\perp})
\end{equation} 
for some $x_0^{**}\in M_{T^{**}}.$ Furthermore, if $X^*$ is assumed to be smooth, then $x_0^{**}$ can be chosen from $M_{T^{**}}\cap E_{X^{**}}$ (see Theorem \ref{th-gen} and Remark \ref{rem-equiv}). In addition, if $X$ is assumed to be reflexive, then we show that the supremum in (\ref{eq-minimax}) is attained at some $x_0\in M_T.$
On the other hand, we show that (\ref{eq-minimax}) holds for an arbitrary closed subspace $Z$ of $Y,$ if we assume that $X^*$ is an $L^1-$predual space (more generally, if $X^{**}$ has $L^1-$property according to Definition \ref{def-l1}), $T$ is a multi-smooth operator of finite order and $T\perp_B \mathcal{L}(X,Z).$ Moreover, we show that in this case, (\ref{equiv}) holds for some $x_0^{**}\in M_{T^{**}}\cap E_{X^{**}}$ (see Theorem \ref{th-propl1}, Remark \ref{rem-equiv} and Remark \ref{rem-l1prop}). Finally, we provide another situation when formula (\ref{eq-minimax}) holds and the supremum is attained at some $x_0\in M_T\cap E_X.$ As a consequence of the results, we prove that if $\mathcal{L}(\ell_1^n,Z)$ is a proximinal subspace of $\mathcal{L}(\ell_1^n,Y),$ then $Z$ is a proximinal subspace of $Y,$ provided  $Z$ is an arbitrary closed subspace of $Y$ and each non-zero element of $Y$ is a multi-smooth point of finite order.   We would like to mention that the approach used in this paper to prove the non-trivial part is completely different from \cite{RAO21}. In particular, we do not use the lifting theorem from \cite[Th. II.2.1]{HWW93}.

\section{Main results}
We begin this section with an easy proposition.
\begin{prop}
	Let $X,Y$ be Banach spaces and $Z$ be a closed subspace of $Y.$ Let $T\in \mathcal{L}(X,Y).$ Suppose there exists $S\in \mathscr{L}_{\mathcal{L}(X,Z)}(T)$ such that $T-S$ is smooth and $M_{T-S}\neq \emptyset.$ Then there exist $x_0\in E_X,~ y_0^*\in E_{Y^*}\cap Z^\perp$ such that the following hold.
	\begin{eqnarray*}
		d(T, \mathcal{L}(X,Z))&=&\sup_{x\in S_X}d(Tx,Z)=d(Tx_0,Z)\\
		                      &=&\sup\{y^*(Tx):x\in E_X,y^*\in E_{Y^*}\cap Z^\perp\}=y_0^*(Tx_0).
	\end{eqnarray*}
\end{prop}
\begin{proof}
	From $S\in \mathscr{L}_{\mathcal{L}(X,Z)}(T),$ it follows that $T_0\perp_B \mathcal{L}(X,Z),$ where $T_0=T-S.$ Since $T_0$ is smooth and $M_{T_0}\neq \emptyset,$ from \cite[Th. 3.3]{SPMR}, we get $M_{T_0}=\{\pm x_0\}$ for some $x_0\in S_X.$ Observe that $x_0\in E_X.$ For otherwise, there exist $x_1,x_2\in B_X$ such that $x_1\neq x_2$ and  $x_0=\frac{1}{2}x_1+\frac{1}{2}x_2.$  Now, from 
	\[\|T_0\|=\|T_0x_0\|\leq \frac{1}{2}\|T_0x_1\|+\frac{1}{2}\|T_0x_2\|\leq \frac{1}{2}\|T_0\|+\frac{1}{2}\|T_0\|=\|T_0\|,\] we get that $\|T_0x_1\|=\|T_0x_2\|=\|T_0\|,$ i.e., $x_1,x_2\in M_{T_0}.$ Thus, $x_1=-x_2$ and so $x_0=0,$ a contradiction. Choose $x^*\in J(x_0).$ Let $z\in Z.$ Consider $A\in \mathcal{L}(X,Z)$ defined as $Ax=x^*(x)z$ for all $x\in X.$ Then $T_0\perp_B A.$ By \cite[Th. 3.3]{SPMR}, we get $T_0x_0\perp_B Ax_0\Rightarrow T_0x_0\perp_B z.$ Since $z\in Z$ is chosen arbitrarily, we have $T_0x_0\perp_B Z.$ Thus, $$d(Tx_0,Z)=d(T_0x_0,Z)=\|T_0x_0\|=\|T_0\|=d(T_0,\mathcal{L}(X,Z))=d(T,\mathcal{L}(X,Z)).$$
	On the other hand,
	\begin{equation}\label{eq-sup}
		\sup_{x\in S_X}d(Tx,Z)\geq d(Tx_0,Z)= d(T,\mathcal{L}(X,Z))\geq \sup_{x\in S_X}d(Tx,Z),
	\end{equation}  where the last inequality follows from (\ref{eq-inequal}). This completes the proof of the first part.\\
    Now, we prove the second part. Since $T_0$ is smooth, again from \cite[Th. 3.3]{SPMR}, we get, $T_0x_0$ is smooth, i.e., $J(T_0x_0)=\{y_0^*\}$ for some $y_0^*\in S_{Y^*}.$ Since $J(T_0x_0)$ is convex, $y_0^*$ is an extreme point of $J(T_0x_0).$ Observe that $J(T_0x_0)$ is an extremal subset of $B_{Y^*}.$ Therefore, $y_0^*\in E_{Y^*}.$ Now, from \cite[Th. 2.1]{J} and  $T_0x_0\perp_B Z,$ it follows that $y_0^*\in Z^{\perp}.$  Thus,
    \begin{equation}\label{eq-sm1}
    	y_0^*(Tx_0)=y_0^*(Tx_0-Sx_0)=y_0^*(T_0x_0)=\|T_0x_0\|=\|T_0\|=d(T,\mathcal{L}(X,Z)).
    	\end{equation}
    On the other hand, observe that for each $x\in E_X,$ $y^*\in E_{Y^*}\cap Z^\perp$ and $A\in \mathcal{L}(X,Z),$ we have
    $$y^*(Tx)=y^*(Tx-Ax)\leq \|Tx-Ax\|\leq\|T-A\|.$$ Thus, 
    \[\sup\{y^*(Tx):x\in E_X,y^*\in E_{Y^*}\cap Z^\perp\}\leq \inf_{A\in \mathcal{L}(X,Z)}\|T-A\|=d(T,\mathcal{L}(X,Z)).\]
    The above inequality together with (\ref{eq-sm1}) completes the proof of the second part.
\end{proof}

To prove the next theorem, we use the extremal structure of the unit ball of $\mathcal{K}(X,Y)^*.$ From \cite[Th. 1.3]{RS}, we note that 
\begin{equation}\label{eq-rs}
	E_{\mathcal{K}(X,Y)^*}=\{x^{**}\otimes y^*:x^{**}\in E_{X^{**}},y^*\in E_{Y^*}\}, 
\end{equation}
where $x^{**}\otimes y^*(S)=x^{**}(S^*y^*)$ for $S\in \mathcal{L}(X,Y).$
Now, we are ready to prove our desired theorem.
\begin{theorem}\label{th-gen}
Let $X,Y$ be Banach spaces. Suppose $Z$ is a subspace of $Y$ such that $Z$ is an $L^1-$predual space and an $M-$ideal in $Y.$ Let $T\in \mathcal{K}(X,Y)$ be a multi-smooth operator of finite order. Suppose that $T\perp_B \mathcal{L}(X,Z).$ Then the following hold.\\
\rm(i) $d(T, \mathcal{L}(X,Z))=\sup_{x\in S_X}d(Tx,Z).$\\
\rm(ii) $T^{**}x_0^{**}\perp_B Z^{\perp \perp}$ for some $x_0^{**}\in M_{T^{**}}.$\\
\rm(iii)   $T^{**}\perp_B \mathcal{L}(X^{**},Z^{\perp\perp}).$\\
\rm(iv) There exists  $x_0^{**}\in M_{T^{**}}$ such that
 \begin{eqnarray*}
	d(T^{**},\mathcal{L}(X^{**},Z^{\perp\perp}))&=&\sup_{x^{**}\in S_{X^{**}}}d(T^{**}x^{**},Z^{\perp\perp})\\
	&=&d(T^{**}x_0^{**},Z^{\perp\perp})=d(T, \mathcal{L}(X,Z)).
\end{eqnarray*}
Additionally, if we assume that $X^*$ is smooth, then in \rm(ii) and \rm(iv), we may choose $x_0^{**}$ from $M_{T^{**}}\cap E_{X^{**}}.$
\end{theorem}
\begin{proof}
	(i) From $T\perp_B \mathcal{L}(X,Z)$ and \cite[Th. 2.1]{J}, it follows that there exists $f\in J(T)$ such that $f(A)=0$ for all $A\in \mathcal{L}(X,Z).$ Since $T$ is a multi-smooth operator of finite order, $Span(E_{J(T)})$ is finite-dimensional. Now, $J(T)$ being a non-empty, weak*compact, convex set, by the Krein-Milman theorem, we get
	\[	J(T)=\overline{conv}^{w^*}(E_{J(T)})\subseteq \overline{Span}^{w^*}(E_{J(T)})=\overline{Span}(E_{J(T)})=Span(E_{J(T)}).\] 
	Using \cite[Lem. 1.1, pp. 166]{S}, we get extreme points $f_1,f_2,\ldots,f_h$ of the unit ball of $Span(E_{J(T)})$ and scalars $\lambda_1,\lambda_2,\ldots, \lambda_h>0$ such that $\sum_{i=1}^h\lambda_i=1$ and $f=\sum_{i=1}^h\lambda_i f_i.$ Now, it is easy to check that $f_i\in E_{J(T)}$ for all $1\leq i\leq h.$ Since $J(T)$ is an extremal subset of $B_{\mathcal{K}(X,Y)^*},$ each $f_i$ is an extreme point of   $B_{\mathcal{K}(X,Y)^*}.$ Therefore, there exist $x_i^{**}\in E_{X^{**}},y_i^*\in E_{Y^*}$ such that $f_i=x_i^{**}\otimes y_i^*$ for each $1\leq i\leq h.$ Now, $f_i\in J(T)$ implies that $$\|T\|=f_i(T)=x_i^{**}\otimes y_i^*(T)=x_i^{**}(T^*y_i^*)\leq \|T^*y_i^*\|\leq \|T^*\|=\|T\|,$$ which yields that $y_i^*\in M_{T^*}$ and $x_i^{**}(T^*y_i^*)=\|T^*y_i^*\|=\|T\|.$
	 Since $Z$ is an $M-$ideal in $Y,$ by \cite[Rem. 1.13, pp. 11]{HWW93} we have $Y^*=Z^*\oplus_1 Z^\perp$ and by \cite[Lem. 1.5, pp. 3]{HWW93} $E_{Y^*}=E_{Z^*}\cup E_{Z^\perp}.$ Thus, $y_i^*\in E_{Z^*}\cup E_{Z^\perp}$ for all $1\leq i\leq h.$ We claim that for some $i,$ $y_i^*\in E_{Z^\perp}.$ If possible, suppose that $y_i^*\in E_{Z^*}$ for all $ i.$ Since $Z$ is an $L^1-$predual space, either $\{y_i^*:1\leq i\leq h\}$ is linearly independent or $y_i^*=\pm y_j^*$ for some $i\neq j.$ In the next two paragraphs, we show that after suitable modification, we can write  $f=\sum_{i=1}^h\lambda_i x_i^{**}\otimes y_i^*,$ where $y_i^*\in E_{Y^*}$ and $\{y_i^*:1\leq i\leq h\}$ is linearly independent.\\

	 Now, suppose that $X^*$ is smooth. Observe that if $y_i^*=y_j^*$ for some $i\neq j,$ then $x_i^{**}(T^*y_i^*)=x_j^{**}(T^*y_i^*)=\|T^*y_i^*\|,$ i.e., $x_i^{**},x_j^{**}\in J(T^*y_i^*).$ The smoothness of $T^*y_i^*$ yields that $x_i^{**}=x_j^{**}.$ In this case, $\lambda_i x_i^{**}\otimes y_i^*+\lambda_j x_j^{**}\otimes y_j^*=(\lambda_i+\lambda_j)x_i^{**}\otimes y_i^*.$ Therefore, in case $X^*$ is smooth, if necessary changing the scalars suitably, we may write $f=\sum_{i=1}^h\lambda_i x_i^{**}\otimes y_i^*,$ where $x_i^{**}\in E_{X^{**}},y_i^*\in E_{Y^*}$ and $\{y_i^*:1\leq i\leq h\}$ is linearly independent. \\

	 Now, suppose that $X^*$ is not smooth. Observe that, if $y_1^*=y_2^*$ holds, then considering $x^{**}=\lambda_1 x_1^{**}+\lambda_2 x_2^{**},$ we get $x^{**}(T^*y_1^*)=(\lambda_1+\lambda_2)\|T^*y_1^*\|$ and $\|x^{**}\|=\lambda_1+\lambda_2.$ In that case, $f=(\lambda_1+\lambda_2)\frac{x^{**}}{\|x^{**}\|}\otimes y_1^*+\lambda_3 x_3^{**}\otimes y_3^*+\ldots+\lambda_h x_h^{**}\otimes y_h^*$ and $x^{**}$ may not belong to $E_{X^{**}}.$ Therefore,  in case $X^*$ is not smooth, if necessary after suitable change, we may write $f=\sum_{i=1}^h\lambda_i x_i^{**}\otimes y_i^*,$ where $x_i^{**}\in S_{X^{**}},y_i^*\in E_{Y^*}$ and $\{y_i^*:1\leq i\leq h\}$ is linearly independent. \\
	 
	 Now,  choose $x^*\in X^*$ such that $x_1^{**}(x^*)\neq 0$
 and $z_0\in \cap_{ i=2}^h\ker(y_i^*)\setminus \ker(y_1^*).$ Define $A\in \mathcal{L}(X,Z)$ by$A(x)=x^*(x)z_0$ for all $x\in X.$ Therefore, $A^*:Z^*\to X^*$ is defined as $A^*z^*=z^*(z_0)x^*$ for all $z^*\in Z^*.$ Now, from $f(A)=0,$ it follows that
 \begin{eqnarray*}
 	\sum_{i=1}^h\lambda_i x_i^{**}\otimes y_i^*(A)=0
 	&\Rightarrow&\sum_{i=1}^h\lambda_i x_i^{**}(A^*y_i^*)=0\\
 		&\Rightarrow&\sum_{i=1}^h\lambda_i x_i^{**}(y_i^*(z_0)x^*)=0\\
 	&\Rightarrow&\sum_{i=1}^h\lambda_i x_i^{**}(x^*)y_i^*(z_0)=0\\
 	&\Rightarrow& \lambda_1 x_1^{**}(x^*)y_1^*(z_0)=0,	
 \end{eqnarray*}
which is a contradiction. This proves our claim. Thus, we get $i\in \{1,2,\ldots, h\}$ such that
\begin{equation}\label{eq-01}
	y_i^*\in E_{Z^\perp}\cap M_{T^*}~\text{and~} x_i^{**}(T^*y_i^*)=\|T\|.
\end{equation}
 Now, for each $x\in X$ and for each $z\in Z,$
\[\|Tx-z\|\geq |y_i^*(Tx-z)|=|y_i^*(Tx)|=|T^*y_i^*(x)|.\]
Thus, taking infimum over $z\in Z,$ we get for each $x\in X,$ $$\|Tx\|\geq d(Tx,Z)\geq |T^*y_i^*(x)|.$$ In this inequality, taking supremum over $x\in S_{X},$ we have
 $$\|T\|\geq \sup_{x\in S_X}d(Tx,Z)\geq \sup_{x\in S_X}  |T^*y_i^*(x)|=\|T^*y_i^*\|=\|T^*\|=\|T\|.$$
 Therefore, $\|T\|=\sup_{x\in S_X}d(Tx,Z).$ On the other hand, from $T\perp_B \mathcal{L}(X,Z),$ it clearly follows that $d(T,\mathcal{L}(X,Z))=\|T\|.$ This proves (i). \\
 
 \noindent (ii) Let $u\in Z^{\perp\perp}.$ Then using (\ref{eq-01}), we get $u(y_i^*)=0.$ Thus, for all $u\in Z^{\perp\perp},$
 \begin{eqnarray*}
 	&&\|T^{**}x_i^{**}-u\|\geq |(T^{**}x_i^{**}-u)y_i^*|=|T^{**}x_i^{**}(y_i^*)|=|x_i^{**}(T^*y_i^*)|\\
 	&\Rightarrow& d(T^{**}x_i^{**},Z^{\perp\perp})\geq |x_i^{**}(T^*y_i^*)|=\|T\|~(\text{taking~ infimum~ over~} u\in Z^{\perp\perp})\\
 	&\Rightarrow&\|T^{**}\|\geq \|T^{**}x_i^{**}\|\geq d(T^{**}x_i^{**},Z^{\perp\perp})\geq \|T\|=\|T^{**}\|\\
 	&\Rightarrow&\|T^{**}\|=\|T^{**}x_i^{**}\|= d(T^{**}x_i^{**},Z^{\perp\perp})\\
 	&\Rightarrow& T^{**}x_i^{**}\perp_B Z^{\perp \perp}~\text{and~} x_i^{**}\in M_{T^{**}}.
 \end{eqnarray*}

\noindent (iii) Let $S\in \mathcal{L}(X^{**},Z^{\perp\perp}).$ Then using (ii), we get
 \[\|T^{**}-S\|\geq \|T^{**}x_0^{**}-Sx_0^{**}\|=\|T^{**}x_0^{**}\|=\|T^{**}\|.\]
 Thus, $T^{**}\perp_B \mathcal{L}(X^{**},Z^{\perp\perp}).$\\
 
 \noindent (iv) From (ii) it follows that $d(T^{**}x_0^{**}, Z^{\perp \perp})=\|T^{**}x_0^{**}\|=\|T^{**}\|$ and from (iii) it follows that $d(T^{**}, \mathcal{L}(X^{**},Z^{\perp\perp}))=\|T^{**}\|.$ Therefore,
 \begin{eqnarray*}
 d(T,\mathcal{L}(X,Z))=	\|T\|=\|T^{**}\|&=&d(T^{**}x_0^{**}, Z^{\perp \perp})\\
 &=&d(T^{**}, \mathcal{L}(X^{**},Z^{\perp\perp}))\\
 	&\geq& \sup_{x^{**}\in S_{X^{**}}}d(T^{**}x^{**}, Z^{\perp \perp})~(\text{similarly~as~}(\ref{eq-inequal}))\\
 	&\geq &d(T^{**}x_0^{**}, Z^{\perp \perp}).
 \end{eqnarray*}
This completes the proof.
\end{proof}
As a simple consequence of Theorem \ref{th-gen}, we get the next corollary.
\begin{cor}
Let $X,Y$ be Banach spaces. Suppose $Z$ is a subspace of $Y$ such that $Z$ is an $L^1-$predual space and an $M-$ideal in $Y.$ Let $T\in \mathcal{K}(X,Y).$ Suppose there exists $S\in \mathscr{L}_{\mathcal{L}(X,Z)}(T)$ such that $T-S$ is a multi-smooth operator of finite order. Then the following hold.\\
\rm(i) $d(T, \mathcal{L}(X,Z))=\sup_{x\in S_X}d(Tx,Z).$\\
\rm(ii) There exists  $x_0^{**}\in S_{X^{**}}$ such that
\begin{eqnarray*}
	d(T^{**},\mathcal{L}(X^{**},Z^{\perp\perp}))&=&\sup_{x^{**}\in S_{X^{**}}}d(T^{**}x^{**},Z^{\perp\perp})\\
	&=&d(T^{**}x_0^{**},Z^{\perp\perp})=d(T, \mathcal{L}(X,Z)).
\end{eqnarray*}
Additionally, if we assume that $X^*$ is smooth, then in \rm(ii), we may choose $x_0^{**}$ from $ E_{X^{**}}.$
\end{cor}
\begin{proof}
	From $S\in \mathscr{L}_{\mathcal{L}(X,Z)}(T)$ it follows that $T-S\perp_B\mathcal{L}(X,Z).$ Now, using Theorem \ref{th-gen} for $T-S$, we get 
	$d(T-S, \mathcal{L}(X,Z))=\sup_{x\in S_X}d(Tx-Sx,Z).$ Since $S\in \mathcal{L}(X,Z),$ we have $d(T-S, \mathcal{L}(X,Z))=d(T, \mathcal{L}(X,Z))$ and $d(Tx-Sx,Z)=d(Tx,Z).$ Therefore, $d(T, \mathcal{L}(X,Z))=\sup_{x\in S_X}d(Tx,Z),$ and thus (i)  holds. Similarly, from (iv) of Theorem \ref{th-gen} we conclude that (ii) holds for some $x_0^{**}\in M_{(T-S)^{**}}\subseteq S_{X^{**}}.$\\
	If $X^*$ is smooth, then from (iv) of Theorem \ref{th-gen} we conclude that (ii) holds for some $x_0^{**}\in M_{(T-S)^{**}}\cap E_{X^{**}}\subseteq E_{X^{**}}.$
\end{proof}

The following corollary shows that in Theorem \ref{th-gen}, if we additionally assume that $X$ is reflexive, then the supremum in (\ref{eq-minimax}) is attained.

\begin{cor}
	Suppose that $X,Y,Z,T$ are as in Theorem \ref{th-gen}. Moreover, assume that $X$ is reflexive. Then  for some $x_0\in M_T,$ the following hold.\\
	(i) $d(T, \mathcal{L}(X,Z))=\sup_{x\in S_X}d(Tx,Z)=d(Tx_0,Z).$\\
	(ii) $Tx_0\perp_B Z^{\perp\perp}.$\\
	(iii) $T\perp_B \mathcal{L}(X,Z^{\perp\perp}).$
\end{cor}
\begin{proof}
From (i) and (iv) of Theorem \ref{th-gen} and the fact that $T^{**}=T$ on $X,$ we get $x_0\in M_T$ such that 
\[d(T, \mathcal{L}(X,Z))=\sup_{x\in S_X}d(Tx,Z)=d(Tx_0,Z^{\perp\perp}).\]
Now, the proof follows from the observation that $Z$ is canonically embedded in $Z^{\perp\perp}$ and therefore,
\[d(Tx_0,Z^{\perp\perp})\leq d(Tx_0,Z)\leq \sup_{x\in S_X}d(Tx,Z).\]
\end{proof}
Note that an important part of Theorem \ref{th-gen} depends on a special property of the extreme points of $Z^*,$ namely if $\{z_1^*,z_2^*,\ldots,z_n^*\}\subseteq E_{Z^*}$ such that $z_i^*\neq \pm z_j^*$ for $1\leq  i\neq j\leq n,$ then the set  $\{z_1^*,z_2^*,\ldots,z_n^*\}$ is linearly independent. Motivated by this property of the extreme points of an $L^1(\mu)$ space, we define this property of a Banach space as $L^1-$property.
\begin{definition}\label{def-l1}
	Let $X$ be a Banach space. We say that $X$ has $L^1-$property if for any given $\{x_1,x_2,\ldots,x_n\}\subseteq E_X,$ with $x_i\neq \pm x_j$ for $1\leq  i\neq j\leq n,$ the set $\{x_1,x_2,\ldots,x_n\}$ is linearly independent.
\end{definition}
Now, we present another situation, where the minimax formula (\ref{eq-minimax}) is satisfied. Most of the arguments of the following theorem are same as in Theorem \ref{th-gen}. For the sake of convenience, we give a sketch of the proof here. 

\begin{theorem}\label{th-propl1}
	Suppose $X$ is a Banach space such that $X^{**}$ satisfies $L^1-$property. Let $Y$ be an arbitrary Banach space and $Z$ be a closed subspace of $Y.$ Let $T\in \mathcal{K}(X,Y)$ be a multi-smooth operator of finite order. Suppose that $T\perp_B \mathcal{L}(X,Z).$ Then the conditions \rm(i)-\rm(iv) of Theorem \ref{th-gen} hold. Moreover, in this case we may choose $x_0^{**}$ of \rm(ii) and \rm(iv) from $E_{X^{**}}\cap M_{T^{**}}.$
\end{theorem}
\begin{proof}
	 As in Theorem \ref{th-gen}, we get $\lambda_i>0,$ $x_i^{**}\in E_{X^{**}},y_i^*\in E_{Y^*}\cap M_{T^*}$ for $1\leq i\leq h$ such that $\sum_{i=1}^h\lambda_i=1,$ $x_i^{**}(T^*y_i^*)=\|T\|$ and $\sum_{i=1}^h\lambda_i x_i^{**}(A^*y_i^*)=0$ for all $A\in \mathcal{L}(X,Z).$ Since $X^{**}$ satisfies $L^1-$property, without loss of generality, we may assume that $\{x_1^{**},\ldots,x_h^{**}\}$ is linearly independent. Choose $x^*\in \cap_{ i=2}^h\ker(x_i^{**})\setminus \ker(x_1^{**}).$ Let $z\in Z$ be arbitrary. Define $A\in \mathcal{L}(X,Z)$ by$A(x)=x^*(x)z$ for all $x\in X.$ Then from  $\sum_{i=1}^h\lambda_i x_i^{**}(A^*y_i^*)=0,$ we get $\lambda_1x_1^{**}(x^*)y_1^*(z)=0.$ Thus, $y_1^*(z)=0.$ Therefore, $y_1^*\in Z^\perp.$ The rest of the proof follows proceeding similarly as in Theorem \ref{th-gen}.
\end{proof}
\begin{remark}\label{rem-l1prop}
	Recall from \cite{L} that a Banach space $W$ is an $L^1-$predual space if and only if $W^{**}$ is isometrically isomorphic to $C(K)$ for some extremally disconnected compact Hausdorff space $K.$ Thus, $W^{**}$ is again an $L^1-$predual space. Hence, in the Theorem \ref{th-propl1} we may consider $X=W^*,$ where $W$ is an $L^1-$predual space. Suppose that, $K$ is a Hyperstonian space and $N(K,\mathbb{R})^+$ is set of all positive normal regular Borel measures on $K.$ Let $N(K,\mathbb{R})=N(K,\mathbb{R})^+-N(K,\mathbb{R})^+.$ Then from Theorem \cite[Th. 10, pp 95]{L}, it follows that $N(K,\mathbb{R})^*$ is isometrically isomorphic to $C(K),$ which is an $L^1-$predual space. Therefore, Theorem \ref{th-propl1} in particular holds for $X=N(K,\mathbb{R}).$ In other words, if $X$ is an abstract $L_1-$space, then $X^*$ is an abstract $M-$space and $X^*=C(K)$ for some compact Hausdorff space $K$ (see \cite[pp. 97]{L}). Since $C(K)$ is an $L_1-$predual space, Theorem \ref{th-propl1} holds for  an abstract $L_1-$ space $X.$ For the definition of abstract $L_1-$space, abstract $M-$space and related results see \cite{L}. 
\end{remark}
\begin{remark}\label{rem-equiv}
	Note that from Theorem \ref{th-gen} (Theorem \ref{th-propl1}), we get the following implications  $$T\perp_B \mathcal{L}(X,Z)\Rightarrow T^{**}x_0^{**}\perp_B Z^{\perp\perp}\Rightarrow T^{**}\perp_B\mathcal{L}(X^{**},Z^{\perp\perp}),$$ for some $x_0^{**}\in M_{T^{**}}.$ Recall from \cite[Prop. 1.11.14, pp 102]{MEGG} that $Z^{**}$ is isometrically isomorphic to $Z^{\perp\perp}.$ Therefore, the space $\mathscr{A}=\{S^{**}:S\in \mathcal{L}(X,Z)\}$ is a subspace of $\mathcal{L}(X^{**},Z^{\perp\perp}).$ Thus, $T^{**}\perp_B\mathcal{L}(X^{**},Z^{\perp\perp})$ yields that $T^{**}\perp_B \mathscr{A},$ i.e., $T^{**}\perp_B S^{**}$ for all $S\in \mathcal{L}(X,Z).$ Now, from the equality $\|T-S\|=\|T^{**}-S^{**}\|,$ we get that $T\perp_B\mathcal{L}(X,Z).$ Therefore, if $X,Y,Z$ and $T$ satisfy the conditions of Theorem \ref{th-gen} (respectively, Theorem \ref{th-propl1}), then we get the following equivalence:
	 $$T\perp_B \mathcal{L}(X,Z)\Leftrightarrow T^{**}x_0^{**}\perp_B Z^{\perp\perp}\Leftrightarrow T^{**}\perp_B\mathcal{L}(X^{**},Z^{\perp\perp}),$$ for some $x_0^{**}\in M_{T^{**}}$ (respectively, $x_0^{**}\in M_{T^{**}}\cap E_{X^{**}}$).	
\end{remark}

In \cite[Prop. 2]{RAO21}, Rao proved that if the supremum in (\ref{eq-minimax}) is attained, then $	d(T^{**},\mathcal{L}(X^{**},Z^{\perp\perp}))=\sup_{x^{**}\in S_{X^{**}}}d(T^{**}x^{**},Z^{\perp\perp})=d(T, \mathcal{L}(X,Z)).$ Note that, Theorem \ref{th-gen} and Theorem \ref{th-propl1} provide us other situations, where such equality holds even if the supremum in (\ref{eq-minimax}) is not attained.\\
Next, we show that if $X$ is reflexive and $Z$ is an arbitrary closed subspace of $Y,$ then some restriction on the norm attainment set of $T$ yields (\ref{eq-minimax}).

\begin{theorem}\label{th-02}
	Let $X$ be a reflexive Banach space and $Y$ be an arbitrary Banach space. Suppose $Z$ is a closed subspace of $Y.$ Let $T\in \mathcal{K}(X,Y)$ be a multi-smooth operator of finite order and $M_T\cap E_X=\{\pm x_i\in S_X:1\leq i\leq n\},$ where $\{x_1,x_2,\ldots,x_n\}$ is linearly independent. Suppose that $T\perp_B \mathcal{L}(X,Z).$ Then there exists $x_0\in M_T\cap E_X$ such that the following hold.\\
	\rm(i) $d(T, \mathcal{L}(X,Z))=\sup_{x\in S_X}d(Tx,Z)=d(Tx_0,Z).$\\
	\rm(ii) $Tx_0\perp_B Z^{\perp\perp}.$\\
	\rm(iii) $T\perp_B \mathcal{L}(X,Z^{\perp\perp}).$
\end{theorem}
\begin{proof}
	Since $T$ is a multi-smooth operator of finite order, following similar arguments as Theorem \ref{th-gen}, we get scalars $\lambda_1,\lambda_2,\ldots,\lambda_h>0$ and $f_1,f_2,\ldots,f_h\in J(T) \cap E_{\mathcal{K}(X,Y)^*}$ such that $\sum_{i=1}^h\lambda_i=1$ and $\sum_{i=1}^h\lambda_if_i(A)=0$ for all $A\in \mathcal{L}(X,Z).$ Since $X$ is reflexive,  from (\ref{eq-rs}) it follows that
	\[E_{\mathcal{K}(X,Y)^*}=\{y^*\otimes x:y^*\in E_{Y^*},x\in E_X\},\]
	where $y^*\otimes x(S)=y^*(Sx)$ for each $S\in \mathcal{L}(X,Y).$ Therefore, for $1\leq i\leq h,$ $f_i=y_i^*\otimes x_i,$ where $y_i^*\in E_{Y^*},x_i\in E_X.$ Now, $y_i^*\otimes x_i\in J(T)$ implies that $x_i\in M_T$ and $y_i^*\in J(Tx_i).$ Without loss of generality, we may assume that $\{x_i:1\leq i\leq h\}$ is linearly independent. Now, choose $x^*\in X^*$ such that $x^*(x_i)=0$ for all $2\leq i\leq h$ and $x^*(x_1)\neq 0.$ Choose an arbitrary $z\in Z.$ Consider $A\in \mathcal{L}(X,Z)$ defined as $Ax=x^*(x)z$ for all $x\in X.$ Then 
	\begin{eqnarray*}
		\sum_{i=1}^h\lambda_if_i(A)=0
	&\Rightarrow&	\sum_{i=1}^h\lambda_iy_i^*\otimes x_i(A)=0\\
	&\Rightarrow& \sum_{i=1}^h\lambda_iy_i^*(Ax_i)=0\\
    	&\Rightarrow& \sum_{i=1}^h\lambda_iy_i^*(z)x^*(x_i)=0\\
    	&\Rightarrow&\lambda_1y_1^*(z)x^*(x_1)=0
    	\Rightarrow y_1^*(z)=0.
	\end{eqnarray*}
Since $z\in Z$ is chosen arbitrarily, we get $y_1^*\in Z^\perp.$ Since for each $u\in Z^{\perp\perp},$  $u(y^*)=0$ holds, we get 
 $$\|Tx_1+u\|\geq |(Tx_1+u)y_1^*|=|y_1^*(Tx_1)|=\|Tx_1\|.$$ Thus, 
$Tx_1\perp_B Z^{\perp\perp}.$ This proves (ii). Moreover, $Tx_1\perp_B Z^{\perp\perp}$ implies that $Tx_1\perp_B Z,$ since $Z$ is canonically embedded in $Z^{\perp\perp}.$  Now, (i) follows from the following inequality:
\begin{eqnarray*}
	d(Tx_1,Z)=\|Tx_1\|=\|T\|=d(T,\mathcal{L}(X,Z))\geq \sup_{x\in S_X}d(Tx,Z)\geq d(Tx_1,Z),
\end{eqnarray*}
where the first inequality follows from  (\ref{eq-inequal}).\\
To prove (iii), let $S\in \mathcal{L}(X,Z^{\perp\perp}).$ 
Observe that 
\[\|T-S\|\geq \|Tx_1-Sx_1\|\geq \|Tx_1\|=\|T\|,\] where the second inequality follows from $Tx_1\perp_B Z^{\perp\perp}.$ Thus $T\perp_B \mathcal{L}(X,Z^{\perp\perp}).$ This completes the proof of the theorem.
\end{proof}
 We immediately get the next corollary due to Theorem \ref{th-02}.

\begin{cor}\label{cor-02}
	Suppose $X=\ell_1^n,$  $Y$ is a Banach space and $Z$ is a closed subspace of $Y.$ Assume that each non-zero element of $Y$ is a multi-smooth point of finite order. Let $ T\in \mathcal{L}(X,Y)\setminus \mathcal{L}(X,Z)$ be such that $\mathscr{L}_{\mathcal{L}(X,Z)}(T)\neq \emptyset.$  Then there exists $x_0\in E_X$ such that
	\[ d(T, \mathcal{L}(X,Z))=\sup_{x\in S_X}d(Tx,Z)=d(Tx_0,Z).\]
	In particular, if $T\perp_B\mathcal{L}(X,Z)$ then $Tx_0\perp_B Z^{\perp\perp},$ for some $x_0\in M_T\cap E_X.$
\end{cor}
\begin{proof}
	Suppose that $(0\neq)~S\in \mathcal{L}(X,Y).$  Since $X=\ell_1^n,$ $M_{S}\cap E_X$ is of the form $\{\pm x_i:1\leq i\leq h\},$ where $\{x_i:1\leq i\leq h\}$ is linearly independent. Moreover, $Sx_i\in Y$ is a multi-smooth point of finite order for each $i.$ Therefore, using  \cite[Cor. 2.3]{MP20} we get $S$ is a multi-smooth operator of finite order. Thus, each non-zero operator of $\mathcal{L}(X,Y)$ satisfies the hypothesis of Theorem \ref{th-02}. Now, suppose that $S_0\in \mathscr{L}_{\mathcal{L}(X,Z)}(T).$ Then $T-S_0\perp_B\mathcal{L}(X,Z).$ Note that, since $T\notin \mathcal{L}(X,Z),T-S_0\neq 0.$ Now, using Theorem \ref{th-02}, we get $x_0\in E_X$ such that 
	\begin{eqnarray*}
		&&d(T-S_0, \mathcal{L}(X,Z))=\sup_{x\in S_X}d(Tx-S_0x,Z)=d(Tx_0-S_0x_0,Z)\\
		&\Rightarrow &d(T, \mathcal{L}(X,Z))=\sup_{x\in S_X}d(Tx,Z)=d(Tx_0,Z).
	\end{eqnarray*}
On the other hand, if $T\perp_B\mathcal{L}(X,Z),$ then from (ii) of Theorem \ref{th-02}, it follows that $Tx_0\perp_B Z^{\perp\perp}$ for some $x_0\in M_T\cap E_X.$ This completes the proof. 
\end{proof}
Note that, in particular, Corollary \ref{cor-02} holds for a smooth Banach space $Y.$ Furthermore, from \cite[Cor. 2.2]{MDP}, it follows that every nonzero operator of $\mathcal{K}(H_1,H_2)$ is a multi-smooth operator of finite-order, where $H_1,H_2$ are Hilbert spaces. For other examples of Banach spaces $Y,$ where each non-zero element is a multi-smooth point of finite order see \cite{KS,LR}.
As a consequence of the results obtained here, we get the following corollary providing sufficient condition for proximinality of a subspace.
\begin{cor}\label{cor-03}
	Suppose $X=\ell_1^n,$  $Y$ is a Banach space and $Z$ is a closed subspace of $Y.$ Assume that each non-zero element of $Y$ is a multi-smooth point of finite order. Suppose that $\mathcal{L}(X,Z)$ is a  proximinal subspace of $\mathcal{L}(X,Y).$ Then $Z$ is a proximinal subspace of $Y.$
\end{cor}
\begin{proof}
	Let $y\in Y\setminus Z.$ Choose $x^*\in E_{X^*}.$ Let $T\in \mathcal{L}(X,Y)\setminus\mathcal{L}(X,Z)$ be defined as $Tx=x^*(x)y$ for all $x\in X.$ From the proximinality of  $\mathcal{L}(X,Z)$ in $\mathcal{L}(X,Y),$ it follows that $\mathscr{L}_{\mathcal{L}(X,Z)}(T)\neq \emptyset.$ Let $S\in \mathscr{L}_{\mathcal{L}(X,Z)}(T).$ Then $T-S\perp_B \mathcal{L}(X,Z).$ Therefore, using Corollary \ref{cor-02}, we get $Tx_0-Sx_0\perp_B Z$ for some $x_0\in E_X.$ Observe that $|x^*(x_0)|=1,$ since $X=\ell_1^n.$ Thus, $x^*(x_0)y-Sx_0\perp_B Z,$ which yields that $\frac{1}{x^*(x_0)}Sx_0\in \mathscr{L}_Z(y).$ This completes the proof.  
\end{proof}
We would like to end the paper with the remark that  Corollary \ref{cor-03} is motivated from \cite[Cor. 5]{RAO21}. However, in \cite[Cor. 5]{RAO21}, one of the assumptions is that $Z$ is an $M-$ideal in $Y,$ which itself is a sufficient condition for the proximinality of $Z$ in $Y$ (see \cite[Prop. 1.1, pp 50]{HWW93}). Here we emphasize that in Corollary  \ref{cor-03}, $Z$ is assumed to be an arbitrary closed subspace of $Y.$

	\bibliographystyle{amsplain}

\end{document}